\documentclass[12pt,dvipsnames]{amsart}
\usepackage{amsmath, amsthm, amscd, amsfonts, amssymb, graphicx, mathabx, bbm, tikz, caption, enumitem, mathtools, comment, pgfplots}
    \pgfplotsset{compat=1.16}
    \usepgfplotslibrary{fillbetween}
    \makeatletter
    \def\l@subsection{\@tocline{2}{0pt}{2.9pc}{5pc}{}}
    \def\l@subsubsection{\@tocline{2}{0pt}{5pc}{7.5pc}{}}
    \makeatother
\usepackage[margin=2.8cm]{geometry}
\usepackage[alphabetic]{amsrefs}

\usepackage[bookmarksnumbered, colorlinks, plainpages]{hyperref}
    \hypersetup{colorlinks=true, linkcolor=MidnightBlue, citecolor=MidnightBlue, filecolor=MidnightBlue, urlcolor=MidnightBlue}
\allowdisplaybreaks
    \mathtoolsset{showonlyrefs,showmanualtags} 
    
\numberwithin{equation}{section}
\newtheorem{thm}[equation]{Theorem}
\newtheorem{lemma}[equation]{Lemma}
\newtheorem{prop}[equation]{Proposition}
\newtheorem{crlr}[equation]{Corollary}
\theoremstyle{definition}

\theoremstyle{remark}
\newtheorem{rmk}[equation]{Remark}

\newcommand{\N}{\mathcal N}
\newcommand{\Z}{\mathbb Z}

\newcommand{\A}{\mathbb A}
\newcommand{\T}{\mathbb T}
\newcommand{\C}{\mathbb C}
\newcommand{\F}{\mathcal F}

\newcommand{\calR}{\mathcal R}
\newcommand{\one}{\mathbbm{1}}

\newcommand{\Abs}[1]{\left\lvert#1\right\rvert}
\newcommand{\jap}[1]{\langle#1\rangle}

\newcommand{\ls}{\lesssim}
\newcommand{\gs}{\gtrsim}
\newcommand{\wb}{\overline}
\newcommand{\eps}{\varepsilon}
\newcommand{\wc}{\widecheck}
\newcommand{\wh}{\widehat}

\begin{document}
\title[Averages with the Gaussian divisor]{Averages with the Gaussian divisor: Weighted Inequalities and the Pointwise Ergodic Theorem}
\author[GIANNITSI]{Christina Giannitsi}
    \address{Department of Mathematics, Vanderbilt University, Nashville, TN, USA}
    \email {christina.giannitsi@vanderbilt.edu}
\author[MIHEISI]{Nazar Miheisi}
    \address{Department of Mathematics, King's College London, London, UK}
    \email{nazar.miheisi@kcl.ac.uk}
\author[MOUSAVI]{Hamed Mousavi}
    \address{Department of Mathematics, University of Bristol, Bristol, UK}
    \email {gj23799@bristol.ac.uk}
\maketitle

\begin{abstract}	
    We discuss the Pointwise Ergodic Theorem for the Gaussian divisor function $d(n)$, that is, for a measure preserving $\Z[i]$ action $T$, the limit 
    $$\lim_{N\rightarrow \infty} \frac{1}{D(N)} \sum _{\mathcal N (n) \leq N} d(n) \,f(T^n x) $$ 
    converges for every $f\in L^p$, where $\mathcal N (n) = n \bar{n}$, and $D(N) = \sum _{\mathcal N (n) \leq N} d(n) $,  and $1<p\leq \infty$. To do so we study the averages 
    $$ A_N f (x) = \frac{1}{D(N)} \sum _{\mathcal N (n) \leq N} d(n) \,f(x-n) ,$$ 
    and obtain improving and weighted maximal inequalities for our operator, in the process. 
\end{abstract}

\tableofcontents
\section{Introduction}

For $1\leq p\leq \infty$, we will say that a function $w:\Z_+\to [0,\infty)$, $w\notequiv 0$,
is \emph{universally $L^p$-good} if for every
measure preserving system $(X,\mu,T)$, and every $f\in L^p(\mu)$, the averages
$$
\frac{1}{W(N)}\sum_{n\leq N} w(n) f(T^n x)
$$
converge almost everywhere, where $W(N) = \sum_{n\leq N} w(n)$.
The classical pointwise ergodic theorem of Birkhoff asserts
that the function $w(n)\equiv1$ is universally $L^1$-good. Establishing this
property for functions with non-trivial arithmetic structure is more challenging.
Here, the archetypal results are the celebrated theorems of Bourgain, and
extended by Weirdl, Mirek and Trojan, showing that the indicator function of the square numbers
or the primes is universally $L^p$-good for all $1<p\leq \infty$ \cite{bourgain1988, bourgain1989, Mirek-Trojan2015, wierdl1988}. Buczolich and Mauldin, and subsequently La Victoire, showed that this range of $p$ is sharp by excluding $p=1$ \cite{Buczolich-Mauldin2010, LaVictoire2011}. We refer to the papers \cite{Boshernitzan-Wierdl1996, Nair1991, Rosenblatt-Wierdl1995} for other results of
this type and alternative presentations of the relevant arguments. An arithmetic
function of particular importance is the divisor function, which counts
the number of positive divisors of an integer $n$. In 2017, Cuny and
Weber proved that it is universally $L^p$-good for all $1<p\leq \infty$ \cite{CW2017}.

\subsection{Main theorems}

The aim of this paper is to study the averaging properties of the divisor function
$d$ in the Gaussian integers $\Z[i]=\{ a +ib \, : \, a,b \in \Z\}$, where
$$
d(n) = \sum_{\substack{d\, : \,d|n}}  1, \quad n \in \Z[i].
$$ 
We begin with some definitions. We say that $(X,\mu,T)$ is a measure
preserving $\Z[i]$-action if $T=(T_1,T_2)$ is a pair of commuting, invertible,
measure preserving transformations of a finite measure space $(X,\mu)$. In this
case, for $n=a+ib\in\Z[i]$ we write $T^n:=T_1^a T_2^b$. As before, we say that $w:\Z[i]\to [0,\infty)$, $w\notequiv 0$, is \emph{universally $L^p$-good} if for every measure preserving
$\Z[i]$-action $(X,\mu,T)$, and every $f\in L^p(\mu)$, the averages
\begin{equation}\label{eq:ergodic-def}
M_N^T(w,f)(x):=\frac{1}{W(N)}\sum_{\N(n)\leq N} w(n) f(T^n x)
\end{equation}
converge as $N\to\infty$ for $\mu$-a.e $x\in X$, where $\N(a+ib) = a^2 + b^2 $
is the algebraic norm of $a+ib$ and $W(N) = \sum_{\N(n)\leq N} w(n)$. Our first
main result is the following:

\begin{thm}\label{t:petdivisor}
The divisor function $d$ is universally $L^p$-good for all $1<p\leq\infty$.
\end{thm}

By the Calder\'{o}n transference principle, Theorem \ref{t:petdivisor} is reduced
to analysis of the averages
\begin{equation}
A_N f (x) = \frac{1}{D(N)} \sum _{\N(n) \leq N} d(n) f(x-n),
\end{equation}
where $D(N) = \sum _{\N(n) \leq N} d(n)$ is the divisor summatory function.
Our other main results, Theorems \ref{thm-improving} and \ref{thm-sparse} below,
give a scale free, $\ell ^p$ improving estimate for these operators, as well as
maximal bounds in weighted $\ell^p$ spaces. These results are in the realm
of discrete harmonic analysis, an area motivated by pointwise ergodic theory
since its introduction. We point to the survey \cite{krause2022} for an
exposition of relevant ideas.

Before stating the results, we need to introduce
some additional notation. For a function $f$ on $\Z[i]$, and a finite set
$E \subset \Z[i]$, we write the $p$-norm average of $f$ on $E$ as
\begin{equation*}
\langle f \rangle _{E,p} = \left( \frac{1}{|E|} \, \sum _{x \in E} |f(x)|^p \right) ^{1/p}.
\end{equation*}
If $E$ is a disk, we let $2E$ be the disk with the same centre and
twice the radius.
Finally, if $1\leq p \leq \infty$ then $p'$ will denote its conjugate
exponent, so that $\dfrac{1}{p} + \dfrac{1}{p'} = 1$.

\begin{thm} \label{thm-improving}
For $1<p<2$, there exists $C_p>0$ such that for all positive integers
$N$, all disks $E$ of radius $\sqrt{N}$, and all functions $f$, 
\begin{equation}
\langle A_Nf \rangle _{E,p'} \leq C_p \,  \langle f \rangle _{2E,p}.
\end{equation}
\end{thm}

A collection $\mathcal S$ of boxes or balls is called sparse if for every $I \in \mathcal S$
there exists a set $E_I \subset I$ so that $|E_I|>|I|/2$ and $E_I$ for $I \in
\mathcal S$ are pairwise disjoint. In addition, the maximal operator associated
to the sequence $A_N$ is defined in the usual way:
$$ A^* f (x) = \sup_N A_Nf(x). $$

\begin{thm} \label{thm-sparse}
	For $1< r,s<2$, the maximal operator $A^*$ is of $(r,s)$-sparse type -- that is, there exists $C>0$ such that for all compactly supported functions $f$ and $g$, there exists a sparse collection $\mathcal S$ so that the $\ell ^2 (\Z[i])$ inner product satisfies
	\begin{equation}
	|(A^*f,g)|\leq C \sum _{I \in \mathcal S}  |I| \langle f \rangle _{I,r} \langle g \rangle _{I,s}
	\end{equation}
\end{thm}

It is well-established that a sparse bound as the one in Theorem \ref{thm-sparse}, implies not only $\ell ^p$  boundedness for the maximal operator, but also weighted inequalities for weights $w$ for which the ordinary Hardy-Littlewood maximal function is bounded.

\begin{crlr}\label{c:maximalestimate}
	For any $1<p<\infty$ and all weights $w$ in the Muckenhoupt class $A_p$, the maximal operator $A^* : \ell ^p (w) \to \ell ^p (w)$ is a bounded operator.
\end{crlr}

The analogues of Theorems \ref{thm-improving} and \ref{thm-sparse} have been
proved for many other averaging operators. Indeed, \cite{giannitsi2022} established
these for the divisor function in the usual integers. However, a particular
novelty of our results is the setting, and to the best of our knowledge the
only other improving and maximal type estimates in the Gaussian integers are
in \cite{gklmr2023}.

\subsection{Overview of our approach}

Our approach is based on the following, well-established idea: we work
on the Fourier side and use the Hardy-Littlewood circle method to analyse
our operators. In particular, we approximate our multiplier $\wh{A_N}$
by one which is localised near the major arcs. In addition, these
approximating multipliers behave like the usual averages on their
supports. The precise statement is Lemma \ref{l:applemm}\ref{l:applemm-power} below.

For the improving and sparse bounds, we use a high/low decomposition.
The \emph{low} part is the main term in the approximation of our multiplier.
We show that this satisfies an $\ell^p\to\ell^\infty$ estimate for all $1<p\leq2$. The \emph{high} part is the error term, which satisfies a
good $\ell^2\to\ell^2$ estimate. These bounds are leveraged in the proof
of Theorems \ref{thm-improving} and \ref{thm-sparse}.

For the pointwise ergodic theorem, our approach is based on ideas of
Bourgain and Lacey (see e.g. \cite{Rosenblatt-Wierdl1995}). Since our
sparse bound implies an $\ell^p$ maximal inequality, it suffices to
prove an oscillation inequality (Theorem \ref{thm:oscillation}).
This is achieved by showing that, in an appropriate sense, $\wh{A_N}$ is
close to a sum of bump functions and then appealing to Bourgain's
multifrequency lemma.

\bigskip

\section{Preliminaries}\label{sec:preliminaries}

We use notation that is standard in the field. Specifically, for two quantities $a$ and $b$, we shall write $a \lesssim b$ if there exists a positive constant $C$ such that $a \leq C \, b$. We shall write $a \lesssim _p b$ if the implied constant depends on $p$. We also write
$a\simeq b$ to mean $a\ls b$ and $b\ls a$.

Throughout the paper we let $e(x)=e^{2\pi i x}$, and for complex numbers $z=x+iy$ and $w=u+iv$, we let $\jap{z,w}=xu+yv$. Then the Fourier
transform of $f:\Z[i]\to\C$ is the function on $\T^2=\C/\Z[i]$ given by
\begin{equation*}
\widehat f (\alpha) = \F f (\alpha) = \sum _{x \in \Z[i]}
f(k)\; e(\langle x ,-\alpha \rangle).
\end{equation*}
In addition, $\widecheck f$ or $\F ^{-1}$ denote the inverse Fourier transform.

Recall that $\Z[i]$ is a Euclidean domain with four unit elements $\pm1, \pm i$. For $a,b,c \in \Z[i]$ we write $(a,b)=c$ to indicate that $c$ is the greatest in norm common divisor of $a$ and $b$ up to a unitary element. The distance of $\beta\in \mathbb{C}$ to the closest point
$n\in \Z[i]$ is denoted $\|\beta\|$. Note that this means that
$\| \beta \| \in [0,1/\sqrt 2]$. We maintain the notation established in \cite{gklmr2023}, where 
\begin{equation}\label{eq-definBq}
    B_q := \{x \in \Z[i] \,:\, 0 \leq \langle x , q\rangle < \N(q) \ \ \text{and} \ \ 0 \leq \langle x , iq\rangle < \N(q) \}.
\end{equation}
Note that for each $q\in \Z[i]$,  $B_q$ creates a tessellation of the complex integer lattice as shown in Figure \ref{fig-Bqexample}.
\begin{figure}
    \centering
    \begin{tikzpicture}[scale=0.7]
        \draw[->, black] (-1,0) -- (3,0);
        \draw[->, black] (0,-0.5) -- (0,6);
        \draw[MidnightBlue, very thick]  (0,0) -- (4,2) ;
        \draw[MidnightBlue, very thick]  (0,0) -- (-2,4) ;
        \draw[dashed,MidnightBlue, very thick]  (4,2) -- (2,6) -- (-2,4) ;
        \draw[black] (4,2) circle (2pt) node[anchor=west]{$q$};
        \draw[black] (-2,4) circle (2pt) node[anchor=east]{$iq$};
        \draw[black] (2,6) circle (2pt) node[anchor=west]{$(1+i)q$};
        \draw[black] (0,5) circle (2pt);
        \draw[black] (3,4) circle (2pt);
        \filldraw[black] (-1,2) circle (2pt);
        \filldraw[black] (-1,3) circle (2pt);
        \filldraw[black] (-1,4) circle (2pt);
        \filldraw[black] (0,0) circle (2pt);
        \filldraw[black] (0,1) circle (2pt);
        \filldraw[black] (0,2) circle (2pt);
        \filldraw[black] (0,3) circle (2pt);
        \filldraw[black] (0,4) circle (2pt);
        \filldraw[black] (1,1) circle (2pt);
        \filldraw[black] (1,2) circle (2pt);
        \filldraw[black] (1,3) circle (2pt);
        \filldraw[black] (1,4) circle (2pt);
        \filldraw[black] (1,5) circle (2pt);
        \filldraw[black] (2,1) circle (2pt); 
        \filldraw[black] (2,2) circle (2pt);
        \filldraw[black] (2,3) circle (2pt);
        \filldraw[black] (2,4) circle (2pt);
        \filldraw[black] (2,5) circle (2pt);
        \filldraw[black] (3,2) circle (2pt);
        \filldraw[black] (3,3) circle (2pt);
    \end{tikzpicture}
    \begin{tikzpicture}[scale=0.7]
        \draw[->, black] (-4,0) -- (8,0);
        \draw[->, black] (0,-0.5) -- (0,10.5);
        \begin{scope}
            \draw[MidnightBlue, very thick]  (0,0) -- (4,2) ;
            \draw[MidnightBlue, very thick]  (0,0) -- (-2,4) ;
            \draw[dashed,MidnightBlue, very thick]  (4,2) -- (2,6) -- (-2,4) ;
            \draw[black] (4,2) circle (2pt) node[anchor=north west]{$q$};
            \draw[black] (-2,4) circle (2pt) node[anchor=east]{$iq$};
            \draw[black] (2,6) circle (2pt);
            \draw[black] (0,5) circle (2pt);
            \draw[black] (3,4) circle (2pt);
            \filldraw[black] (-1,2) circle (2pt);
            \filldraw[black] (-1,3) circle (2pt);
            \filldraw[black] (-1,4) circle (2pt);
            \filldraw[black] (0,0) circle (2pt);
            \filldraw[black] (0,1) circle (2pt);
            \filldraw[black] (0,2) circle (2pt);
            \filldraw[black] (0,3) circle (2pt);
            \filldraw[black] (0,4) circle (2pt);
            \filldraw[black] (1,1) circle (2pt);
            \filldraw[black] (1,2) circle (2pt);
            \filldraw[black] (1,3) circle (2pt);
            \filldraw[black] (1,4) circle (2pt);
            \filldraw[black] (1,5) circle (2pt);
            \filldraw[black] (2,1) circle (2pt); 
            \filldraw[black] (2,2) circle (2pt);
            \filldraw[black] (2,3) circle (2pt);
            \filldraw[black] (2,4) circle (2pt);
            \filldraw[black] (2,5) circle (2pt);
            \filldraw[black] (3,2) circle (2pt);
            \filldraw[black] (3,3) circle (2pt);
        \end{scope}
        \begin{scope}[shift={(4,2)}]
            \draw[MidnightBlue, very thick]  (0,0) -- (4,2) ;
            \draw[MidnightBlue, very thick]  (0,0) -- (-2,4) ;
            \draw[dashed,MidnightBlue, very thick]  (4,2) -- (2,6) -- (-2,4) ;
            \draw[black] (4,2) circle (2pt) node[anchor=west]{$2q$};
            \draw[black] (-2,4) circle (2pt) node[anchor=south]{$(1+i)q$};
            \draw[black] (2,6) circle (2pt) node[anchor=south]{$(2+i)q$};
            \draw[black] (0,5) circle (2pt);
            \draw[black] (3,4) circle (2pt);
            \filldraw[black] (-1,2) circle (2pt);
            \filldraw[black] (-1,3) circle (2pt);
            \filldraw[black] (-1,4) circle (2pt);
            \filldraw[black] (0,0) circle (2pt);
            \filldraw[black] (0,1) circle (2pt);
            \filldraw[black] (0,2) circle (2pt);
            \filldraw[black] (0,3) circle (2pt);
            \filldraw[black] (0,4) circle (2pt);
            \filldraw[black] (1,1) circle (2pt);
            \filldraw[black] (1,2) circle (2pt);
            \filldraw[black] (1,3) circle (2pt);
            \filldraw[black] (1,4) circle (2pt);
            \filldraw[black] (1,5) circle (2pt);
            \filldraw[black] (2,1) circle (2pt); 
            \filldraw[black] (2,2) circle (2pt);
            \filldraw[black] (2,3) circle (2pt);
            \filldraw[black] (2,4) circle (2pt);
            \filldraw[black] (2,5) circle (2pt);
            \filldraw[black] (3,2) circle (2pt);
            \filldraw[black] (3,3) circle (2pt);
        \end{scope}
        \begin{scope}[shift={(-2,4)}]
            \draw[MidnightBlue, very thick]  (0,0) -- (4,2) ;
            \draw[MidnightBlue, very thick]  (0,0) -- (-2,4) ;
            \draw[dashed,MidnightBlue, very thick]  (4,2) -- (2,6) -- (-2,4) ;
            \draw[black] (4,2) circle (2pt);
            \draw[black] (-2,4) circle (2pt) node[anchor=east]{$2iq$};
            \draw[black] (2,6) circle (2pt) node[anchor=west]{$(1+2i)q$};
            \draw[black] (0,5) circle (2pt);
            \draw[black] (3,4) circle (2pt);
            \filldraw[black] (-1,2) circle (2pt);
            \filldraw[black] (-1,3) circle (2pt);
            \filldraw[black] (-1,4) circle (2pt);
            \filldraw[black] (0,0) circle (2pt);
            \filldraw[black] (0,1) circle (2pt);
            \filldraw[black] (0,2) circle (2pt);
            \filldraw[black] (0,3) circle (2pt);
            \filldraw[black] (0,4) circle (2pt);
            \filldraw[black] (1,1) circle (2pt);
            \filldraw[black] (1,2) circle (2pt);
            \filldraw[black] (1,3) circle (2pt);
            \filldraw[black] (1,4) circle (2pt);
            \filldraw[black] (1,5) circle (2pt);
            \filldraw[black] (2,1) circle (2pt); 
            \filldraw[black] (2,2) circle (2pt);
            \filldraw[black] (2,3) circle (2pt);
            \filldraw[black] (2,4) circle (2pt);
            \filldraw[black] (2,5) circle (2pt);
            \filldraw[black] (3,2) circle (2pt);
            \filldraw[black] (3,3) circle (2pt);
        \end{scope}
    \end{tikzpicture}
    \caption{\small Tessallation of the plane by $B_q$}
    \label{fig-Bqexample}
\end{figure}
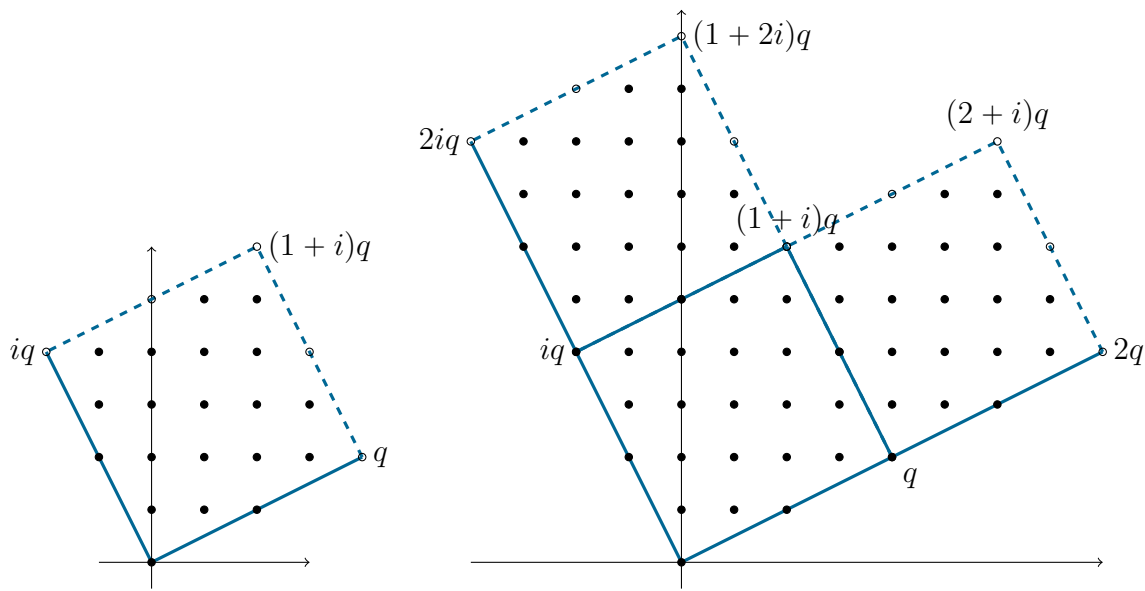
Moreover, the set of elements that are co-prime with $q$ is denoted by
\begin{equation}
    \A_q : =  \{ r \in B_q \, : \, (r,q)=1 \}.
\end{equation}

For a positive integer $k$, let $r_2(k)$ denote the number of representations
of $k$ as the sum of two squares. Clearly, $r_2(k)=|\{n\in\Z[i]: 
\N(n)=k\}|$. It is well-known that $r_2(k)$ obeys the following
asymtptotics:
\begin{equation}
\sum_{k=1}^N r_2(k) = \pi N + O(N^{1/2}),
\qquad
\sum_{k=1}^N \frac{r_2(k)}{k} = \pi(\log N + \kappa) + O(N^{-1/2}),
\label{eq:sum-of-squares}
\end{equation}
where $K=\pi\kappa$ is Sierpi\'{n}ski's constant.
Clearly the first sum above is the number of lattice point in $Z[i]$ of norm at most $N$. It is easy to see that an analogous asymptotic
holds for the number of lattice points in a sector. To be precise, let
$\omega\subseteq [0,2\pi]$ be an interval and let $\Gamma_N(\omega)$
be the sector
$$
\Gamma_N(\omega) = \{z: \N(z)\leq N, \; \arg(z)\in\omega\}.
$$
Then the number of lattice points in $\Gamma_N(\omega)$ is
$$
\Big| \Z[i]\cap \Gamma_N(\omega) \Big|
=
\frac{|\omega|}{2} N + O(N^{1/2}).
$$
We also need the following estimate for exponential sums on a sector:

\begin{lemma}\label{lem:lattice}
Let $\omega\subseteq[0,2\pi]$ be a an interval. Then for all $N\geq1$
we have
\begin{equation}
\frac{1}{2\pi}\int_0^{2\pi} \Abs{\sum_{n\in \Gamma_N(\omega + t)}
e(\jap{n,\alpha})} \;dt
\ls
N^{1/2} + N^{1/4}\|\alpha\|^{-3/2}.
\end{equation}
The implied constant is independent of $\omega$.
\end{lemma}

\begin{proof}
We can suppose that $\N(\alpha)=\|\alpha\|^2\leq 1/2$.
For $n\in\Z[i]$, let $I_n$ be the square in $\C$ centred at $n$ with
side length $1$. Observe that
\begin{align}
    \sum_{n\in \Gamma_N(\omega)} \int_{I_n} e(\jap{\xi,\alpha})\;d\xi
    & = 
    \int_{I_0} e(\jap{\xi,\alpha})\;d\xi
    \sum_{n\in \Gamma_N(\omega)} e(\jap{n,\alpha}) \\ 
    & = \wc{\one_{I_0}}(\alpha)\sum_{n\in \Gamma_N(\omega)} e(\jap{n,\alpha}).
\end{align}
Note that $\wc{\one_{I_0}}(\alpha)\simeq 1$ for $\N(\alpha)\leq1/2$.
Also, the measure of the symmetric difference of $\Gamma_N(\omega)$
and $\bigcup_{n\in \Gamma_N(\omega)}I_n$ is $O(N^{1/2})$. It follows that 
\begin{align}
\Abs{\sum_{n\in \Gamma_N(\omega + t)} e(\jap{n,\alpha})}
&\simeq
\Abs{\sum_{n\in \Gamma_N(\omega + t)} \int_{I_n} e(\jap{\xi,\alpha})\;d\xi} \\
&\leq
\Abs{\int_{\Gamma_N(\omega + t)} e(\jap{\xi,\alpha})\;d\xi} + O(N^{1/2}) \\
&=
N\Abs{\int_{\Gamma_1(\omega + t)} e(\jap{\xi,N^{1/2}\alpha})\;d\xi} + O(N^{1/2}).
\end{align}
Moreover, for each $\beta\in\C$,
$$
\frac{1}{2\pi}\int_0^{2\pi} \int_{\Gamma_1(\omega + t)} e(\jap{\xi,\beta})\;d\xi\;dt
=
\int_0^1 \int_{\Gamma_1(\omega)} e(\jap{\xi,\beta_t})\;d\xi\;dt,
$$
where $\beta_t=|\beta|e(t)$, and so it suffices to show that the last
expression is $O(\N(\beta)^{-3/4})$. Indeed, if this is the case, then
by the calculations above we would have
\begin{align}
    \frac{1}{2\pi}\int_0^{2\pi} \Abs{\sum_{n\in \Gamma_N(\omega + t)}
    e(\jap{n,\alpha})} \;dt
    & \ls N^{1/2} + N\N\left(N^{1/2}\alpha\right)^{-3/4}\\
    & = N^{1/2} + N^{1/4}\|\alpha\|^{-3/2}.
\end{align}

Let $\gamma$ be the boundary of $\Gamma_1(\omega)$. Then by the divergence theorem
\begin{equation}
\int_{\Gamma_1(\omega)} e(\jap{\xi,\beta_t})\;d\xi
=
\frac{i}{2\pi \N(\beta)} \int_\gamma
e(\jap{\xi,\beta_t}) \jap{\nu(\xi),\beta_t} \; d\sigma(\xi),
\label{eq:divergence}
\end{equation}
where $\nu(x)$ is the outward unit normal at $\xi$ and $d\sigma$ is the arc length
measure on $\gamma$. Observe that $\gamma$ is a disjoint union
$\gamma_1\sqcup\gamma_2\sqcup\gamma_3$, where $\gamma_1$ and $\gamma_2$ are
straight lines from $0$ and $\gamma_3$ is an arc of the unit circle. We
can suppose $\gamma_1$ is the interval $[0,1]$. Then
\begin{align}
    \Abs{\int_{\gamma_1} e(\jap{\xi,\beta_t}) \jap{\nu(\xi),\beta_t} \; d\sigma(\xi)}
    & \leq |\beta| \Abs{\int_0^1 e(x|\beta|\cos(2\pi t)) \;dx}\\
    & \ls |\beta| \Abs{\frac{\sin(\pi|\beta|\cos(2\pi t))}{\pi|\beta|\cos(2\pi t)}}.
\end{align}
Clearly the same is true for $\gamma_2$ and so
\begin{equation}
\int_0^1 \left|\int_{\gamma_1\cup\gamma_2}
e(\jap{\xi,\beta_t}) \jap{\nu(\xi),\beta_t} \; d\sigma(\xi) \right| \;dt
\ls
|\beta|\int_0^1
\Abs{\frac{\sin(\pi|\beta|\cos(2\pi t))}{\pi|\beta|\cos(2\pi t)}} \;dt
\ls
\log|\beta|.
\label{eq:line}
\end{equation}

We now consider the integral over $\gamma_3$.
For $\xi\in\gamma_3$, write $\xi=e(\theta)$ with $2\pi\theta\in\omega$. Then
$\nu(\xi)=\xi$ and $\jap{\xi,\beta_t}=|\beta|\cos(2\pi(\theta-t))$.
It follows that
\begin{equation}
\int_{\gamma_3}
e(\jap{\xi,\beta_t}) \jap{\nu(\xi),\beta_t} \; d\sigma(\xi)
=
|\beta| \int_{\omega+t}
e(|\beta|\cos2\pi\theta) \cos2\pi\theta\; d\theta.
\label{eq:arc}
\end{equation}
By splitting the integral on the right into integrals over subintervals
of $\omega+t$ on which $\cos2\pi\theta$ is monotonic and using the change
of variable $2\pi u=\cos2\pi\theta$,
we see that the right hand side is controlled by a sum of at most 3 terms of the form
$$
|\beta|\left|\int_a^b e(|\beta|u)\frac{u}{\sqrt{1-u^2}}\;du \right|,
$$
with $a,b\in[-1,1]$. This is a constant multiple of
$|\beta||(\wc{\one}_{[a,b]}\ast J_1)(|\beta|)|$, where $J_1$ is a Bessel function of the
first kind. Since $\wc{\one}_{[a,b]}(x)=O(x^{-1})$, the implied constant being independent
of $a$ and $b$, and $J_1(x)=O(x^{-1/2})$, we conclude that \eqref{eq:arc} is
$O(|\beta|^{1/2})$. Combining this with \eqref{eq:divergence} and
\eqref{eq:line} then gives the result.
\end{proof}

\begin{rmk}\label{rmk:expsum}
Taking $\omega=[0,\pi]$ in Lemma \ref{lem:lattice} we get the estimate
$$
\Abs{\sum_{\N(n)\leq N}
e(\jap{n,\alpha})}
\ls
\min\left\{N, \;N^{1/2} + N^{1/4}\N(\alpha)^{-3/4}\right\}.
$$
Then by summation by parts, it is not difficult to see that
\begin{align}
\Abs{\sum_{\N(n)\leq N}\log\N(n) e(\jap{n,\alpha})}
&\ls
\log N\Abs{\sum_{\N(n)\leq N} e(\jap{n,\alpha})}
+ \Abs{\sum_{k=1}^N \frac{1}{k} \sum_{\N(n)\leq k} e(\jap{n,\alpha})} \\
&\ls
\min\left\{N, \; (N^{1/2} + N^{1/4}\N(\alpha)^{-3/4})\log N \,\right\}.
\end{align}
\end{rmk}

\bigskip

\section{Exponential sums with the divisor function in $\Z[i]$}
\label{sec:expsum}

Throughout this section, we set
$$ S_N = D(N)A_N = \sum_{\N(n)\leq N} d(n)\; \one_n. $$
Let $\omega\subseteq[0,2\pi]$
be an interval and $0\leq t <2\pi$. It will also be convenient to define
$$ S_{N,t}^\omega = \sum_{\substack{\N(n)\leq N\\\arg(n)\in\omega+t}} d(n)\; \delta_n. $$

\subsection{Rational frequencies}
We begin by using the Dirichlet hyperbola method to evaluate
$\wh{S_{N,t}^{\omega}}$ at rational points.

\begin{lemma}\label{lem:rationals}
Let $\omega\subseteq[0,2\pi]$ be an interval and $0\leq t <2\pi$. For
each $a,q\in\Z[i]$ with $q\neq0$ and $a\in\A_q$, and each integer $N\geq1$,
\begin{equation}\label{e:rationals}
\wh{S_{N,t}^{\omega}}(a/q)
= \frac{|\omega|\pi N}{2\N(q)}\Big(\log N - 2\log \N(q) + 2\kappa -1 \Big)
+ E_N^\omega(a/q,t),
\end{equation}
where $E_N^\omega(a/q,t)$ satisfies
$$
\frac{1}{2\pi}\int_0^{2\pi} \left|E_N^\omega(a/q,t)\right| \;dt 
\ls (N^{3/4} + N^{1/4}\N(q)^{3/4})\log(1+\N(q)).
$$
The implied constant is independent of $\omega$.
\end{lemma}

\begin{rmk}\label{rmk:rationals}
Inspection of the proof shows that in the case that $\N(q)=1$ and $a=0$,
the error estimate above holds for all $t$ uniformly rather than just on average. In particular, for any interval $\omega$ we have the formula
$$
\sum_{\substack{\N(n)\leq N\\\arg(n)\in\omega}} d(n)
=
\frac{|\omega|\pi N}{2}\left(\log N + 2\kappa -1 \right)
+
O(N^{3/4}).
$$
\end{rmk}

\begin{proof}
For each $n\in\Z[i]\setminus\{0\}$ and $t\in\T$, set
\begin{align}
    D_1(n,t) & = \sum_{\substack{1\leq \N(m)\leq N/\N(n)\\ \arg(mn)\in\omega+t}} e(\jap{mn,a/q}), \\
    D_2(n,t) & = \sum_{\substack{1\leq \N(m)\leq \sqrt{N}\\ \arg(mn)\in\omega+t}} e(\jap{mn,a/q}),
\end{align}
so that 
\begin{equation}
   \wh{S_{N,t}^{\omega}}(a/q)
        = \sum_{\substack{\N(mn)\leq N\\ \arg(mn)\in\omega+t}}
        e(\jap{mn,a/q})
        = \sum_{1\leq \N(n)\leq \sqrt{N}} (2D_1(n,t) - D_2(n,t)).
\end{equation}

We consider the cases $\N(q)\leq\sqrt{N}$ and $\N(q)>\sqrt{N}$ separately.

\textbf{Case 1:} $\N(q)\leq\sqrt{N}$.
If $n\neq0$ is a multiple of $\wb{q}$, then
\begin{align}
    D_1(n,t) & = \frac{|\omega| N}{2\N(n)}
        + O\left(\frac{N^{1/2}}{\N(n)^{1/2}}\right).
\end{align}
Then summing each term on the right over all such $n$, using
\eqref{eq:sum-of-squares}, and applying Cauchy-Schwartz, we get
\begin{align}
    \sum_{\substack{1\leq \N(n)\leq \sqrt{N} \\ n\equiv 0\mod \wb{q}}} \frac{|\omega| N}{2\N(n)}
    & = \frac{|\omega| N}{2\N(q)} \sum_{k=1}^{\sqrt{N}/\N(q)} \frac{r_2(k)}{k}\\
    & = \frac{|\omega|\pi N}{4\N(q)}\left(\log N - 2\log \N(q) + 2\kappa\right) + O(N^{-1/2}),
\intertext{and}
    \sum_{\substack{1\leq \N(n)\leq \sqrt{N} \\ n\equiv 0\mod \wb{q}}} \frac{N^{1/2}}{\N(n)^{1/2}}
    & \leq \sum_{1\leq \N(n)\leq \sqrt{N}} \frac{N^{1/2}}{\N(n)^{1/2}}\\
    & \leq \sum_{1\leq \N(n)\leq \sqrt{N}} \frac{N^{1/2}}{\N(n)^{1/2}}
    \ls N^{3/4}.
\end{align}

Combining these we find that
\begin{align}
    \sum_{\substack{1\leq \N(n)\leq \sqrt{N} \\ n\equiv 0\mod \wb{q}}} D_1(n,t)
    & = \frac{|\omega|\pi N}{4\N(q)}\left(\log N - 2\log \N(q) + 2\kappa\right) + O(N^{3/4}).
\end{align}
The same reasoning also gives
\begin{align}
    \sum_{\substack{1\leq \N(n)\leq \sqrt{N} \\ k\equiv 0\mod \wb{q}}} D_2(n,t)
    & = \frac{|\omega|\pi N}{2\N(q)} + O(N^{3/4}),
\end{align}
and hence
$$
E_N^\omega(a/q,t)
=
\sum_{\substack{1\leq \N(n)\leq \sqrt{N} \\ n\notequiv 0\mod \wb{q}}}
(2D_1(n,t) - D_2(n,t))
+ O(N^{3/4}).
$$

If $\wb{n}\equiv s \mod q$ for some $s\in B_{q}\setminus\{0\}$, then 
$ \N(n) \gs  \|sa/q\|^2$. Then by Lemma \ref{lem:lattice},
\begin{align}
\frac{1}{2\pi}\int_0^{2\pi} \left| D_1(n,t) \right| \;dt
& =   \frac{1}{2\pi}\int_0^{2\pi}
\left| \sum_{\substack{1\leq \N(m)\leq N/\N(n)\\ \arg(mn)\in\omega+t}}
e(\jap{m,sa/q}) \right| \;dt \\
&\ls
\frac{N^{1/2}}{\N(n)^{1/2}}
+ \frac{N^{1/4}}{\N(n)^{1/4}\|sa/q\|^{3/2}} \\
&\leq 
\frac{N^{1/2}}{\N(n)^{1/2}} + \frac{N^{1/4}}{\|sa/q\|^2}.
\end{align}
As before, we have that
\begin{align}
    \sum_{\substack{1\leq \N(n)\leq \sqrt{N} \\ \wb{n}\notequiv 0\mod q}} \frac{N^{1/2}}{\N(n)^{1/2}}
       & \ls N^{3/4}.
\end{align}
Note that the map $s\mapsto sa \mod q$ is a bijection of $B_q$, and so
\begin{align}
    \sum_{s\in B_q\setminus\{0\}} 
    \sum_{\substack{1\leq \N(n)\leq \sqrt{N} \\ \wb{n}\equiv s\mod q}}
    \frac{N^{1/4}}{\|sa/q\|^2}
         &\ls \sum_{s\in B_q\setminus\{0\}} \frac{N^{3/4}}{\N(q)\|sa/q\|^2} \\
         &\leq \sum_{s\in B_q\setminus\{0\}} \frac{N^{3/4}}{\N(s)} 
        \ls N^{3/4}\log(1+\N(q)).
\end{align}    

We conclude that
$$
\sum_{\substack{1\leq \N(n)\leq \sqrt{N} \\ n\notequiv 0\mod \wb{q}}}
\frac{1}{2\pi}\int_0^{2\pi} \left| D_1(n,t) \right|\;dt
\ls
N^{3/4}\log(1+\N(q)).
$$
The same computations give an analogous bound for $D_2(n,t)$,
which concludes the proof in the case that $\N(q)\geq\sqrt{N}$.

\textbf{Case 2:} $\N(q)>\sqrt{N}$.
In this case there are no non-zero $n$ such that $\N(n)\leq\sqrt{N}$ and $n$ is a multiple of $\wb{q}$. However,
\begin{align}
    \Abs{\frac{|\omega|\pi N}{4\N(q)}\left(\log N - 2\log \N(q) + 2\kappa-1\right)}
        & \ls N^{1/2}\log(1+\N(q)),
\label{eq:large-q}
\end{align}
and so
$$
E_N^\omega(a/q,t) = \wh{S_{N,t}^{\omega}}(a/q) +O(N^{1/2}\log(1+\N(q)))
$$

Observe that for each $s\in B_q$ the number of $n$ such that
$\N(n)\leq\sqrt{N}$ and $\wb{n}\equiv s\mod q$ is at most 4. It follows from
this and H\"{o}lder's inequality that
\begin{align*}
\sum_{1\leq \N(n)\leq \sqrt{N}}  \frac{1}{2\pi}\int_0^{2\pi}
\left | D_1(n,t) \right| \;dt 
	&\ls
    \sum_{s\in B_q\setminus\{0\}}
    \sum_{\substack{1\leq \N(n)\leq \sqrt{N} \\ \wb{n}\equiv s\mod q}}
    \frac{N^{1/4}}{\N(n)^{1/4}\|sa/q\|^{3/2}} \\
        &\ls \sum_{s\in B_q\setminus\{0\}} \frac{N^{1/4}}{\N(s)^{1/4}\|sa/q\|^{3/2}} \\
        &\ls N^{1/4} \log(1+\N(q))^{1/4} \left(\sum_{s\in B_q\setminus\{0\}} \frac{1}{\|sa/q\|^2}\right)^\frac{3}{4}\\
        & \ls N^{1/4}\N(q)^{3/4}\log(1+\N(q)).
\end{align*}
The same holds for $D_2(n,t)$. Combining these completes the proof.
\end{proof}

\medskip 
\subsection{Major and minor arcs}\label{sec:circlemethod}

Let $\delta>0$ be a small constant and let $a,q\in\Z[i]$ be
such that $q\neq0$ and $a\in\A_q$. Then for
$N\geq1$, we define
\begin{equation}\label{e:majorarcsdefinition}
\mathfrak{M}_{N,\delta}(a/q)
=
\left\{ \alpha\in \T^2 \, :\,
\N(\alpha-a/q) \leq N^{-1+\delta/2}\N(q)^{-1} \right\}.
\end{equation}

The \emph{major arcs} $\mathfrak{M}_{N,\delta}$ consist of
$\mathfrak{M}_{N,\delta}(a/q)$ with $\N(q)\leq N^\delta$ and the
\emph{minor arcs}  $\mathfrak{m}_{N,\delta}$ are the complement of
the major arcs:
$$
\mathfrak{M}_{N,\delta}=\bigcup_{\N(q)\leq N^\delta\,}
\bigcup_{\;a\in \A_q}\mathfrak{M}_{N,\delta}(a/q),
\qquad
\mathfrak{m}_{N,\delta}
= \T^2\setminus\mathfrak{M}_{N,\delta}.
$$
The next lemma describes the behaviour of $\wh{S_N}$ on major arcs.

\begin{prop}\label{p:themajorarcgauss}
Let $0<\delta\leq 1/20$ and $N\geq1$. For each $a,q\in\Z[i]$ such that
$1\leq\N(q)\leq N^\delta$ and $a\in\A_q$, and each
$\alpha\in\mathfrak{M}_{N,\delta}(a/q)$ we have

\begin{equation}\label{e:themajorarcgauss}
\wh{S_N}(\alpha)
= \sum_{\substack{\N(n)\leq N}} \frac{\pi}{\N(q)}
\left(\log \N(n) -2\log{\N(q)} +2\kappa\right)
e(\jap{n, \alpha-a/q}) + O(N^{1-\delta}).
\end{equation}
\end{prop}

\begin{proof}
Let $\mathcal P$ be a partition of $[0,2\pi]$ into intervals of equal
length $\lambda \simeq N^{-2\delta}$. For each integer $j\geq0$, set
$N_j = (1+\lambda)^j\sqrt{N}$ and let $j_0$ be such that $N_{j_0}
\leq N<N_{j_0+1}$. For $0<t<2\pi$, $P\in \mathcal P$ and
$j=0,1\dots,j_0$ we set
$$
R (j, P, t) = \Big\{ n \;\colon\;   N_j \leq \N(n) < N _{j+1},\; 
\arg(n)\in P+t\Big\}.
$$
Note that
$$
|R(j,P,t)| = \frac{\lambda^2}{2}N_{j-1} + O\left(N_{j}^{1/2}\right)
\quad\text{and}\quad
\sum_{n\in R(j,P,t)}\,d(n) \ls N^{-4\delta}N_j\log N_j.
$$
Choose $\beta =\beta(P,t)\in\C$ with $\N(\beta)=1$ and
$\arg(\beta)\in P+t$. Then for each $n\in R(j,P,t)$,
$$
\left|1-e(\jap{n-\beta\sqrt{N_j}, \alpha-a/q})\right|
\leq
\N(n-\beta\sqrt{N_j})^{1/2} \; \N(\alpha-a/q)^{1/2}
\ls
N^{ - \delta},
$$
and so
\begin{align}\label{eq:major1}
D_{j,t}^P(\alpha)
:&= \sum_{n\in R(j,P,t)} d(n) e(\jap{n,\alpha}) \\
&= D_{j,t}^P(a/q)\; e(\jap{\beta\sqrt{N_j}, \alpha-a/q})
 + O( N^{-5\delta}N_j).
\end{align}

Next, for each $j$ set
$$
\Delta_j :=
\frac{\lambda\pi N_j}{2\N(q)}
\left(\log N_j - 2\log \N(q)+2\kappa-1\right).
$$
Then for each $j\geq1$ we have
\begin{align}
\Delta_j-\Delta_{j-1}
&=
\frac{\lambda\pi}{2\N(q)}
\left[(N_j-N_{j-1})(\log N_j - 2\log \N(q)+2\kappa-1)
+ N_{j-1}\log(1+\lambda)\right] \\
&=
\frac{\lambda\pi}{2\N(q)}
(N_j-N_{j-1})(\log N_j - 2\log \N(q)+2\kappa) +
O(N^{-6\delta} N_j)\\
&=
\frac{\pi}{\N(q)}
\sum_{n\in R(j,P,t)} (\log \N(n) - 2\log\N(q) + 2\kappa)
+ O(N^{-6\delta} N_j).
\end{align}
It follows that
\begin{align}
\label{eq:major2}
&\;\;\;\; \sum_{j=1}^{j_0}  (\Delta_j-\Delta_{j-1})
e(\jap{\beta\sqrt{N_j}, \alpha-a/q}) \\
&\qquad =
\frac{\pi}{\N(q)}\sum_{j=1}^{j_0}\sum_{n\in R(j,P,t)}
(\log \N(n) - 2\log\N(q) + 2\kappa)
e(\jap{\beta\sqrt{N_j}, \alpha-a/q})
+ O(N^{1-4\delta}) \\
&\qquad =
\frac{\pi}{\N(q)}\sum_{j=1}^{j_0}\sum_{n\in R(j,P,t)}
(\log \N(n) - 2\log\N(q) + 2\kappa - 1)e(\jap{n, \alpha-a/q})
+ O(N^{1-3\delta}) \\
&\qquad =
\frac{\pi}{\N(q)}
\sum_{\substack{\N(n)\leq N \\ \arg(n)\in P+t}}
(\log \N(n) - 2\log\N(q) + 2\kappa - 1)e(\jap{n, \alpha-a/q})
+ O(N^{1-3\delta}).
\end{align}
Since $|\wh{S_{\sqrt{N},t}^P}(\alpha)|\ls N^{1/2}\log N$, using
\eqref{eq:major1}, we can write
\begin{align}
\label{eq:major3}
\wh{S_{N,t}^P}(\alpha)
&= \sum_{j=1}^{j_0}
D_{j,t}^P(a/q)e(\jap{\beta\sqrt{N_j}, \alpha-a/q})
+ O(N^{1-3\delta})\\
&= \sum_{j=1}^{j_0}
(D_{j,t}^P(a/q) - \Delta_j + \Delta_{j-1}) 
e(\jap{\beta\sqrt{N_j}, \alpha-a/q}) \\
&\qquad 
+ \sum_{j=1}^{j_0} (\Delta_j-\Delta_{j-1})
e(\jap{\beta\sqrt{N_j}, \alpha-a/q}) +O(N^{1-3\delta}).
    \end{align}
Let $E_{N}^P(a/q,t)$ be as in Lemma \ref{lem:rationals}. Then using
summation by parts, we can bound the first term on the right as follows:
\begin{align}
& \Abs{\sum_{j=1}^{j_0}
(D_{j,t}^P(a/q) - \Delta_j + \Delta_{j-1}) 
e(\jap{\beta\sqrt{N_j}, \alpha-a/q})} \\
&\qquad\qquad =
\Abs{\sum_{j=1}^{j_0}
[E_{N_j}^P(a/q,t) - E_{N_{j-1}}^P(a/q,t)]
e(\jap{\beta\sqrt{N_j}, \alpha-a/q})}\\
&\qquad\qquad \leq
\Abs{\sum_{j=1}^{j_0} E_{N_j}^P(a/q,t)
(1-e(\jap{\beta(\sqrt{N_j}-\sqrt{N_{j+1}}), \alpha-a/q}))}\\
&\qquad\qquad\qquad\qquad 
+ |E_{N_{j_0}}^P(a/q,t)| + |E_{\sqrt{N}}^P(a/q,t)| \\
&\qquad\qquad \leq
N^{-\delta}\sum_{j=1}^{j_0}|E_{N_j}^P(a/q,t)|
+ |E_{N_{j_0}}^P(a/q,t)| + |E_{\sqrt{N}}^P(a/q,t)|.
\end{align}
Combining this with \eqref{eq:major2} and \eqref{eq:major3}, then
summing over all $P\in\mathcal{P}$ gives
\begin{align}
&\left| \wh{S_N}(\alpha) - 
 \sum_{\substack{\N(n)\leq N}} \frac{\pi}{\N(q)} \left(\log \N(n) -2\log{\N(q)} +2\kappa\right) e(\jap{n, \alpha-a/q}) \right| \\
&\qquad\qquad \ls
N^{-\delta}\sum_{P\in\mathcal{P}}\sum_{j=1}^{j_0}|E_{N_j}^P(a/q,t)|
+ \sum_{P\in\mathcal{P}}|E_{N_{j_0}}^P(a/q,t)| + \sum_{P\in\mathcal{P}}|E_{\sqrt{N}}^P(a/q,t)| + N^{1-\delta}
\end{align}
Finally, averaging over $t$ and applying Lemma \ref{lem:rationals} gives the
conclusion.
\end{proof}

We now provide estimates for the complementary minor arcs:

\begin{prop}\label{p:minorexpineq}
Assume that $\alpha\in \mathfrak{m}$. Then 
$$
\left|\wh{S_N}(\alpha)\right| \ls N^{1-\delta/4}\log N.
$$
\end{prop}

\begin{proof}
By Dirichlet's Approximation Theorem there exist $a,q\in\Z[i]$ with $a\in\A_{q}$ such that
$\N(\alpha-a/q)\leq N^{-1+ \delta/2}\N(q)^{-1}$. Since $\alpha\in\mathfrak{m}$,
we have $\N(q)>N^\delta$, and so $\N(\alpha-a/q)\leq N^{-1- \delta/2}$.
Then
$$
\left|\wh{S_N}(\alpha) - \wh{S_N}(a/q)\right|
\ls
\sum_{\N(n)\leq N} d(n) \N(n)^{1/2}\N(\alpha-a/q)^{1/2}
\leq
N^{1-\delta/4}\log N.
$$
\end{proof}

\medskip 

\section{Approximation}

\subsection{Approximating multipliers}
We will establish an approximation result for the kernel
of our operator, on the Fourier side, by utilizing Section
\ref{sec:circlemethod}. To this end, we introduce multipliers
$L_{N,q}$ defined by
\begin{align}\label{eq:approximatingkernel}
    \widehat {L_{N,q}}(\alpha)
    = \frac{1}{N\log N}\frac{\pi}{\N(q)}\sum_{\substack{\N(n)\leq N}} \left(\log \N(n) -2\log{\N(q)} +2\kappa\right) e(\jap{n, \alpha}).
\end{align}
Using the estimates in Remark \ref{rmk:expsum}, one easily sees that for
all $\alpha$ we have the bound
\begin{equation}\label{eq:kernelestimate}
|\widehat {L_{N,q}}(\alpha)| \ls \min \left\{\; 1, \N(q)^{-1}(N^{-1/2} + N^{-3/4}\N(\alpha)^{-3/4}) \right\}.
\end{equation}
We will also need the Fourier kernel of the usual averaging operators:
\begin{align}
    \widehat M_{N}(\alpha) 
    &= \frac{1}{\pi N}\sum_{\substack{\N(n)\leq N}}  e(\jap{n,\alpha}).
\end{align}
Next, we introduce a cut-off by considering a Schwartz function $\eta$
such that
$$ \one_{\N(x) \leq \frac{1}{16}} \leq \eta \leq \one_{\N(x)
\leq \frac{1}{8}}.
$$
Then for an integer $s\geq0$, we set $\eta_s(\alpha) = \eta(2^s \alpha)$ and
$$
\calR_s = \{a/q: \; a\in\A_q, \; 2^s \leq \N(q) < 2^{s+1} \}.
$$
An important observation that we will use frequently is that the
functions $\eta_s(\cdot - a/q)$, $a/q\in\calR_s$, have pairwise
disjoint supports. We can now define our approximating multipliers
$K_{N,s}$ and $K'_{N,s}$ as follows:
\begin{align}
\wh{K_{N,s}}(\alpha)
&=
\sum_{a/q\in\calR_s} 
\widehat{L_{N,q}}(\alpha-a/q) \, \eta_{s}(\alpha-a/q);\\
\wh{K'_{N,s}}(\alpha)
&=
\sum_{a/q\in\calR_s} \frac{\pi^2}{\N(q)}
\widehat{M_N}(\alpha-a/q) \, \eta_{s}(\alpha-a/q). \label{eq:KNdef}
\end{align}

\begin{lemma}\label{l:applemm}
For each $0<\delta\leq 1/20$, the following estimates hold:
\begin{enumerate}[label=(\roman*), itemsep=15pt, topsep=15pt]
\item\label{l:applemm-power}
$\displaystyle
\left\| \wh{A_N} - \sum_{s:2^s \leq N^\delta}\wh{K_{N,s}} \right\|_\infty
\ls N^{-\delta/4};
$
\item\label{l:applemm-log}
$\displaystyle
\left\| \wh{A_N} - \sum_{s:2^s \leq N^\delta}\wh{K'_{N,s}} \right\|_\infty
\ls \frac{1}{\log N}.
$
\end{enumerate} 
\end{lemma}

\begin{proof}
\ref{l:applemm-power}
Fix $\alpha\in\T^2$ and $N$ large. By Dirichlet's Approximation Theorem there
exist $a_0,q_0\in\Z[i]$ with $a_0\in\A_{q_0}$ and  $\N(q_0) \leq N^{1-\delta/2}$
such that $\N(\alpha-a_0/q_0)\leq N^{-1+ \delta/2}\N(q_0)^{-1}$. Suppose $2^{s_0}\leq\N(q_0)<2^{s_0+1}$. We consider to cases depending on $s_0$.

\textbf{Case 1:} $2^{s_0} \leq N^\delta$. In this case $\eta_{s_0}(\alpha-a_0/q_0)=1$ and
so $\wh{K_{N,s_0}}(\alpha)=\widehat{L_{N,q_0}}(\alpha-a_0/q_0)$. Then by Proposition
\ref{p:themajorarcgauss},
\begin{equation}
|\wh{A_N}(\alpha) - \widehat{L_{N,q_0}}(\alpha-a_0/q_0)| \ls N^{-\delta}.
\label{eq:case1-main}
\end{equation}
For any $s\neq s_0$ and $a/q\in\calR_s$,
$$
\N(\alpha-a/q)
\gs
\N(a/q - a_0/q_0) - \N(\alpha-a_0/q_0)
\geq
\N(1/qq_0)
\gs
N^{-\delta}\N(q)^{-1}.
$$
Using \eqref{eq:approximatingkernel}, and noting that $\delta$ is small,
we see that
$$
|\widehat {L_{N,q}}(\alpha)| \ls \N(q)^{-1}(N^{-1/2} + N^{-3/4}N^{3\delta/4}\N(q)^{3/4})
\ls
2^{-s/4} N^{-\delta},
$$
and so
\begin{equation}
\left |\sum_{\substack{s:2^s \leq N^\delta \\ s\neq s_0}}\wh{K_{N,s}}(\alpha)\right|
\ls N^{-\delta}.
\label{eq:case1-error}
\end{equation}
Combining \eqref{eq:case1-main} and \eqref{eq:case1-error} completes the proof in this case.

\textbf{Case 2:} $2^{s_0} > N^\delta$. Then for any $s$ such that $2^s\leq N^\delta$ and $a/q\in\calR_s$, $\N(\alpha - a/q) \gs N^{-1+\delta/2}\N(q)^{-1}$. The argument above
shows that 
\begin{equation}
\left |\sum_{s:2^s \leq N^\delta}\wh{K_{N,s}}(\alpha)\right|
\ls N^{-\delta}.
\end{equation}
In addition, by Proposition \ref{p:minorexpineq}, we have that $|\wh{A_N}(\alpha)|\leq N^{-\delta/4}$.
Combining these completes the proof of \ref{l:applemm-power}.

\ref{l:applemm-log}
This follows from \ref{l:applemm-power} upon observing that for each $s$,
$$ \|\wh{K_{N,s}}(\alpha) - \wh{K'_{N,s}}(\alpha)\|_\infty
\ls \frac{1}{\log N}. $$
\end{proof}

\medskip 
\subsection{The High and Low Parts}
We now decompose $A_N$ into \emph{high} and \emph{low} parts as follows:
for $\delta>0$, we write
$$
\wh{A_N} = \wh{\textup{Lo}_{N,\delta}} + \wh{\textup{Hi}_{N,\delta}},
\quad\text{where}\quad
\widehat{\textup{Lo}_{N,\delta}} (\alpha) =  \sum_{s:2^s \leq N^\delta}\wh{K_{N,s}}(\alpha).
$$
Lemma \ref{l:applemm}\ref{l:applemm-power} states that for $\delta$
sufficiently small,
\begin{equation}\label{eq:highpart}
\|\wh{\textup{Hi}_{N,\delta}}\|_{\infty}\ls N^{-\delta/4}.
\end{equation}
This is sufficient for our purposes and so we turn our attention to the
low part. Notice that 
\begin{align}
\textup{Lo}_{N,\delta} (n)
    &= \sum_{s:\,2^s\leq N^{\delta}}
    \sum_{a/q\in\calR_s}
    \mathcal F^{-1} \left(
    \widehat {L_{N,q}}(\cdot - a/q) \;
    \eta_{s}(\cdot-a/q) \right)(n) \\
    & = \sum_{s:\,2^s\leq N^{\delta}}
    \sum_{a/q\in\calR_s}
    \Big(L_{N,q} * \wc {\eta_{s}}\Big)(n) \, e(\jap{n,a/q}) \\
    & = \sum_{\N(q)<2N^\delta}
    \Big(L_{N,q} * \wc{\eta_{s}}\Big)(n) \, \tau_q(n)
\end{align}
where $\tau_q(n)$ is a \emph{Ramanujan sum}
$$ \tau_q(n) := \sum_{a\in\A_q} e(\jap{n,a/q}). $$
Thus to control the Low part, one needs to first be able to control Ramanujan sums. The following result can be easily adapted from \cite{gklmr2023}*{Prop 3.11}.

\begin{lemma}\label{lem:Ramanujan}
For all $\eps>0$ and all integers $N,Q,k>2$ such that $N>Q^k$, we have
$$
\left(\frac{1}{N}\sum_{\N(n)\leq N}
\left(\sum_{\N(q)<Q} \frac{|\tau_q(n)|}{\N(q)}\right)^k\right)^{1/k}
\ls_{k,\eps}
Q^\eps.
$$
\end{lemma}

Using this, we can now prove $\ell^p\rightarrow \ell^\infty$ bound
for $\textup{Lo}_{N,\delta}$.

\begin{lemma}\label{lem:low-part}
Let $1<p<2$, let $N>2$ be an integer and let $E$ be a disk of radius $\sqrt{N}$.
Then for all $\delta\leq 1/p'$ and all functions $f$ supported on $2E$, we have
$$
\jap{\textup{Lo}_{N,\delta} f}_{E,\infty}
\ls_p
N^{\delta/p'}\jap{f}_{2E,p}.
$$
\end{lemma}

\begin{proof}
It suffices to consider the case when $p'$ is an integer.

We begin by observing that
$$
|(L_{N,q}\ast \wc\eta_{s})(x)| \leq \|L_{N,q}\|_{\ell^\infty} 
 \|\wc\eta_{s}\|_{\ell^1} \ls \frac{1}{N \N(q)}.
 $$
Consequently, for $x\in E$ we have that 
\begin{align*}
|\textup{Lo}_{N,\delta} f(x)| &\ls \sum_{\N(y)<4N}\sum_{\N(q)<2N^\delta} |\tau_q(y)|  L_{N,q}\ast \wc\eta_{s}(y)| |f(x-y)| \\
    &\ls \frac{1}{N}\sum_{\N(y)< 9N}\sum_{\N(q)<2N^\delta} \frac{|\tau_q(y)|}{\N(q)}|f(x-y)| \\
    &\leq \left(\frac{1}{N}\sum_{\N(n)<9N} \left(\sum_{\N(q)<2N^\delta} \frac{|\tau_q(n)|}{\N(q)}\right)^{p'}\right)^{1/p'} \left(\frac{1}{N}\sum_{\N(y)<9N} |f(x-y)|^p\right)^{1/p}. 
\end{align*}
The conclusion now follows from Lemma \ref{lem:Ramanujan} with
$Q=\lfloor 2N^\delta \rfloor$, $k=p'$ and $\eps = 1/p'$.
\end{proof}

\section{Proof of the main theorems}

As is expected for these types of questions, \cites{hkly2020,giannitsi2022,glmr2022,gklmr2023}, 
a multifrequency type of inequality is required, especially for the sparse bound. One can easily follow the steps of Lacey's proof of Bourgain's inequality \cite{lacey1997} to obtain the following:

\begin{thm}\label{t:multi}
Let $\phi:\mathbb{C}\rightarrow \mathbb{C}$ be a smooth function that satisfies:
\begin{align*}
    |\phi(x)| \lesssim |x|^{-3} , &&
    |\hat{\phi}(\alpha)-1| \lesssim |\alpha| ,&&
    |\hat{\phi}(\alpha)|  \lesssim |\alpha|^{-2}
\end{align*}
Assume that $K\in \Z[i]$ and $\N(K)>1$. Let $\lambda_1,\cdots, \lambda_L$ be in $2^{-j_0}K^{-1}\Z[i]$. Then 
\begin{align}
    \left\lVert \sup_{j\geq j_0} 
    \left\vert\sum_{\ell=1}^K 
    \textup{Mod}_{\lambda_{\ell}} \,\left[\phi_j\ast 
    \textup{Mod}_{-\lambda_{\ell}} f\right] \right \vert \right \rVert_{\ell^2} \lesssim \log\log (L+\N(K)) \lVert f\rVert_{\ell^2},
\end{align}
where $\phi _j (x) = 2^{-j}\phi (2^{-j}x)$ and $\textup{Mod}_{\theta} f(x) = e(\theta \cdot x)f(x)$.
\end{thm}

\subsection{Fixed scale estimate}

First we focus on a fixed scale estimate. 
It has been well-established in the literature that it suffices to obtain a
restricted weak-type estimate for all $p$ sufficiently close to $1$. Then
the result follows by real interpolation. (see e.g. \cite[Chapter 1
\S2]{krause-book2023}). Thus, the proof of Theorem \ref{thm-improving} boils down to the following Lemma.

\begin{lemma}\label{l:lemfixscale}
Let $1<p<2$, let $N\geq1$ be an integer and let $E$ be a disk of radius $\sqrt{N}$. Then for $f=\one_F$ supported on $2E$ and $g=\one_G$ supported on $E$, we have\begin{equation}\label{eq:improving}
\frac{1}{N}|( A_N f,g )|\lesssim \langle f\rangle_{2E,p}\langle g\rangle_{E,p}.
\end{equation}
\end{lemma}
\smallskip 

\begin{rmk}
    We emphasize that Lemma \ref{l:lemfixscale} fails for $p=1$, making the endpoint sharp. Indeed, for $N\geq 4$, let $M=\lfloor \log^2 N \rfloor$ and set $f(x) = \one_{[M,2M]}\big(d(x)\big)$ and $g(x) = \one_0(x)$. Then one can easily verify that
    \begin{align}
        \frac{1}{N}|( A_N f,g )| 
        & \simeq \frac{1}{N^2\log N}\left |\sum_{\N(n)\leq N} f(n)d(n) \right | \\
        & > \frac{\log N}{N^2} \sum_{\N(n)\leq N} |f(n)|\\
        & = \log N \;  \langle f\rangle_{2E,1}\langle g\rangle_{E,1} .
    \end{align}
\end{rmk}

\begin{proof}
Take $1<q<2$ and $0<\eps\leq 1/100q'^2$. Note that if
\begin{equation}\label{eq:simplebound}
\jap{f}_{2E,q'}\jap{g}_{E,q'}\leq N^{-\eps},
\end{equation}
then using a trivial estimate for $A_Nf$, we get that for all $p\geq q$,
$$
\frac{1}{N}(A_N f,g)
\ls
N^\eps \jap{f}_{2E,1}\jap{g}_{E,1}
\leq
\jap{f}_{2E,p}\jap{g}_{E,p}.
$$
Now consider the case when \eqref{eq:simplebound} fails, i.e when
\begin{equation}\label{eq:lowerbound}
\min\{\jap{f}_{2E,1}, \jap{g}_{E,1}\} > N^{-\eps q'}.
\end{equation}
Then using \eqref{eq:highpart} and Lemma \ref{lem:low-part} we see that for
all $0<\delta<\min\{1/20, 1/q'\}$,
\begin{align}\label{eq:highlow}
\frac{1}{N}(A_N f,g)
&=
\frac{1}{N}( \textup{Hi}_{N,\delta} f,g)
+ \frac{1}{N}(\textup{Lo}_{N,\delta} f,g) \\
&\ls
N^{-\delta/4} \jap{f}_{2E,1}^{1/2} \jap{g}_{E,1}^{1/2}
+ N^{\delta/q'} \jap{f}_{2E,1}^{1/q} \jap{g}_{E,1}.
\end{align}
The two terms on the right are equal if we choose $\delta$ such that
$$
N^{\delta(q'+4)/4q'} = \jap{f}_{2E,1}^{1/2 - 1/q} \jap{g}_{E,1}^{-1/2}.
$$
Observe that this choice of $\delta$ is within the allowed range. Indeed,
\eqref{eq:lowerbound} implies that
$$
\delta \leq \frac{4\eps q'^2}{q(q'+4)} < \frac{1}{25q'}.
$$
Then in this case, the right hand side of \eqref{eq:highlow} becomes
$$
2\left(\jap{f}_{2E,1}\jap{g}_{E,1}\right)^{\frac{q'+2}{q'+4}}
\jap{f}_{2E,1}^{\frac{-1}{q'+4}}
\leq
\left(\jap{f}_{2E,1}\jap{g}_{E,1}\right)^{\frac{q'+1}{q'+4}},
$$
which gives \eqref{eq:improving} for $p=(q'+4)/(q'+1)$. Then letting $q\to1$
concludes the proof.
\end{proof}

\medskip 
\subsection{Sparse bounds}
Traditionally, analogous proofs in the literature depend upon the Kesler, Lacey, Mena recursive argument from \cite{KLM2020}*{Theorem 1.2}, that allows them to reduce the proof of a sparse domination result to a simpler estimate for indicators. Our reduced version takes the following form.
\begin{lemma}\label{sparselm} 
Fix $E=\{ x \in \Z[i]\;:\; \N(x) \leq N \}$. For all admissible stopping times $T$, and for all $1<p<2$ we have that
\begin{equation} \label{eq:stoppingtimeestimate}
 (  A_T f, g  )  \ls \big[ \langle f \rangle _{E,1} \; \langle g \rangle _{E,1} \big] ^{1/p} \, |E|,
\end{equation}
where a stopping time $T$ is called \textit{admissible} if and only if for any $I \subseteq E$ such that 
$\langle f \rangle _{3I,1} > 100 \langle f \rangle _{E,1}$,  
we have $\inf \{T(x) \, : \, x \in I \} > |I|$.
\end{lemma}

Before we discuss the proof of Lemma \ref{sparselm}, we shall revisit the proof of Kelser, Lacey and Mena, to verify that it still holds in our adjusted setting. More specifically, 

\begin{lemma}
Suppose that for every $E=\{ x \in \Z[i]\,:\, \N(x) \leq N \}$, and all admissible stopping times $T$, a family of operators $K_N$ satisfies the estimate in equation \eqref{eq:stoppingtimeestimate}. Then there exists a sparse family of balls $\mathcal S$ so that the operators are dominated by the corresponding $(p,p)$-sparse form.
\end{lemma}
\begin{proof}
    Pick indicator functions $f$ and $g$ supported on a fixed ball $E$. Let $\mathcal Q_E$ be the collection of all balls $Q$ that are maximally contained inside our ball $E$, for which we have $\langle f \rangle _Q > C \langle f \rangle _E$. We have for our admissible stopping times:
    \begin{align*}
        \Big(  \sup _{n \leq |E|} K_n f,g \Big)  
            & \leq \left (   K_T f,g \right )  + \sum _{Q \in \mathcal Q_E} \Big( \sup _{n \leq |Q|} K_n f\one_Q,g\one_Q \Big) 
    \end{align*}
    
    The first term above satisfies our desired bound by assumption. 
    Note that, because of the admissibility condition and the fact that $f$ is a non-negative function, we must have 
    $$ \sum _{Q\in \mathcal Q_E} |Q| = \left|\bigcup_{Q\in \mathcal Q_E} Q \right| \leq \frac{1}{2}|E|.$$

    For the second term, one can repeat the argument for each ball $Q$. Using this recursion, we can deduce the desired sparse domination result, with the associated sparse family be the collection of all possible $\mathcal Q_E$'s. The sparse condition is satisfied thanks to our observation for the size of $Q$ above, and the respective exceptional set defined in the standard way $\mathcal E_E=E\setminus \mathcal Q_E$.
\end{proof}
\smallskip 

\begin{proof}[Proof of Lemma \ref{sparselm}]
Note that H\"older, along with a trivial bound produces
\begin{align}
(A_T f, g)
    & \leq \langle A_T f\rangle _{E,\infty} \langle g\rangle _{E,1}\, |E| 
    \leq \sup_{x\in E} \big\{T^\delta (x) \big\} \; \langle f\rangle _{E,1} \, \langle g\rangle _{E,1} |E|.
\end{align}
for $0<\delta < \frac{1}{100(p'+1)}$. Let $B_J=\Big\{ x\in \Z[i] \::\: T^\delta (x) \ls J^{\frac{1}{p}} \leq  \langle f \rangle ^{-\frac{1}{p'}} _{E,1} \; \langle g \rangle ^{-\frac{1}{p'}} _{E,1} \Big\}$, and notice that when restricted on $B_J$, \eqref{eq:stoppingtimeestimate} holds trivially. So we devote the rest of the proof in studying what happens when restricting to $x\notin B_J$.

We follow a slight modification of the fixed scale approach. Specifically, this time decompose our operator into three pieces $\widehat A_Tf \ls \widehat H_{T,1}f + \widehat H_{T,2}f + \widehat L_{T,0}f$, and provide their estimates separately. Once again we rely on \eqref{eq:highpart} and Lemma \ref{lem:low-part}. In the end we bring them together and optimize our parameters as in the fixed scale proof.

First, we look at $\widehat H_{T,1} := \widehat{\operatorname{Hi}}_{T,3p'\,\delta}$, where the inequalities below come from \eqref{eq:highpart}, an elementary square function argument, and Parseval
\begin{align}
\langle  H_{T,1}*f\rangle_{E,2} ^2
    \ls \sum _{\N(n)>J^{\frac{1}{p'\;\delta}}} n^{-2p'\,\delta} \langle f \rangle_{E,2}^2
    \ls J^{-1+\frac{1}{p'}} \langle f \rangle_{E,2}^2.
\end{align}

Next, we define the $H_{T,2}$ term as
\begin{align}
\widehat H_{T,2} := \sum_{k=\lfloor \log K \rfloor}^{\log J}
    \sum_{\substack{ 2^k\leq  N(q)< 2^{k+1} 
    }}
    \sum_{a\in \mathbb{A}_q}
    \widehat L_{N,q}(\alpha - \frac{a}{q}) \;
    \eta_{s}(\alpha-\frac{a}{q}),
\end{align}
and appeal to Theorem \ref{t:multi} for the next estimate
\begin{align}
    \left\| \sup_N \left| 
    \sum_{\substack{ 2^k\leq  N(q)< 2^{k+1} 
    }}
    \big( L_{N,q} * \wc\eta_{s} * f \big) \; \tau_q 
    \right| \right\|_{\ell^2} 
    \ls k 2^{-k} \|f\|_{\ell^2} .
\end{align}
Thus, summing over $k$ and using Parseval, we can verify that 
\begin{align} 
    \langle H_{T,2} *f \rangle_{E,2} 
    \ls J^{-1+\frac{1}{p'}} \; \langle f\rangle_{E,2} 
\end{align}

Finally, we define our last term
\begin{align}
    \widehat L_{T,0} := \widehat{\operatorname{Lo}}_{K,3p'\,\delta},
\end{align}
and, using Lemma \ref{lem:low-part}, one can easily verify that 
\begin{align}
\langle L_{T,0}*f\rangle _{E,\infty} & \ls J^{1/p'} \langle f \rangle ^{1/p} _{E,1}			
\end{align}	

Putting everything together, we have shown that 
\begin{align}
    \langle A_Tf \rangle _{E,2} &\ls J^{-1+\frac{1}{p'}} \langle f\rangle_{E,2} \\
    \langle A_Tf \rangle _{E,\infty} &\ls J^{\frac{1}{p'}} \langle f\rangle_{E,p} 
\end{align}
Last, we optimize with respect to $J$, so that 
\begin{align}
    J^{-1+\frac{1}{p'}} \; \langle f \rangle ^{1/2} _{E,1} \; \langle g \rangle ^{1/2} _{E,1}
    \sim J^{\frac{1}{p'}} \; \langle f \rangle ^{1/p} _{E,1} \; \langle g \rangle _{E,1}
    \quad \Longrightarrow \quad 
    J\sim \langle f \rangle _{E,1} ^{\frac{1}{2}-\frac{1}{p'}}\; \langle g \rangle _{E,1} ^{-\frac{1}{2}}
\end{align}
Since this is well-within the allowed range, we have established the desired estimate.
\end{proof}

\subsection{The Pointwise Ergodic Theorem}
Let $(X,\mu,T)$ be a measure preserving $\Z[i]$-action and recall
the definition of $M_N^T=M_N^T(d,f)$ in \eqref{eq:ergodic-def}.
For $R\geq1$, set $I_R = \{\,\lfloor 2^{k/R} \rfloor:\; k\in\Z_+\,\}$. The
following Lemma shows that in order to establish Theorem \ref{t:petdivisor},
it suffices to prove that for each $R\in\Z_+$, $M_N^Tf$ converges as
$N\to\infty$ with $N\in I_R$.

\begin{lemma}\label{lem:lacunary}
Let $f\geq0$ be a measurable function on $X$. Suppose that for each $R\in\Z_+$,
the sequence $\{M_N^Tf: N\in I_R\}$ converges almost everywhere. Then the
sequence $\{M_N^Tf: N\in \Z_+\}$ converges almost everywhere.
\end{lemma}

\begin{proof}
We begin by observing that each sequence $I_R$ contains $\{2^k:\;k\in\Z_+\}$
as a subsequence and so the limit of $M_N^Tf$ as $I_R \ni N\to\infty$
does not depend on $R$; denote this limit by $L$.

Write $\rho_k=\lfloor 2^{k/R} \rfloor$. For $N\geq1$, choose $k$
such that $\rho_k \leq N < \rho_{k+1}$. Then one can easily check that
$$ \frac{D(\rho_k)}{D(\rho_{k+1})}\, M_{\rho_k}^T f
    \leq M_N^Tf
    \leq \frac{D(\rho_{k+1})}{D(\rho_{k})} M_{\rho_{k+1}}^T f. $$
Since $D(\rho_{k+1})/D(\rho_k)\to 2^{1/R}$ as $k\to\infty$, we conclude that
$$ 2^{-1/R} L
\leq \liminf_{N\to\infty} M_N^Tf 
\leq \limsup_{N\to\infty} M_N^Tf
\leq 2^{1/R} L. $$
Letting $R\to\infty$ completes the proof.
\end{proof}

For $f\in L^p(\mu)$, let $M_*^Tf$ be the maximal function
$$ M_*^Tf(x) = \sup_{N\geq1} \Abs{M_N^Tf(x)}. $$
In light of the Calder\'{o}n Transference Principle \cite{calderon, emerson74}, Corollary
\ref{c:maximalestimate}
implies that $M_*^T$ is bounded on $L^p(\mu)$ for all $1<p<\infty$. It then
follows, by the Banach Principle, that the set
$$ \left\{f\in L^p(\mu) :\; \lim_{N\to\infty} M_N^Tf \;\;\text{exists $\mu$-a.e}\right\} $$
is closed in $L^p(\mu)$. Thus it remains show this set is dense. We do
this with the following oscillation inequality.

\begin{thm}\label{thm:oscillation}
For each $R\in\Z_+$, each sequence of integers $\{N_k\}_{k\geq1}$ such that
$N_{k+1}\geq2 N_k$, and $f\in L^2(\mu)$, we have
\begin{align}\label{e:amineq}
\sum_{k=1}^\infty \left\|
\max_{\substack{N_k\leq N < N_{k+1} \\ N\in I_R}}
|M_N^Tf - M_{N_{k+1}}^Tf|\right\|_{L^2(\mu)}^2
\ls_R
\|f\|_{L^2(\mu)}^2.
\end{align}
\end{thm}

Firstly, let us note that Theorem \ref{thm:oscillation} implies the almost
everywhere convergence of $M_N^Tf$ for all $f\in L^2(\mu)$. Indeed if
$M_N^Tf$ does not converge $\mu$-a.e, there must be some $\eps>0$, $R\geq1$,
and a sequence $\{N_k\}_{k\geq1}$ with $N_{k+1}\geq2 N_k$ such that for all
$k\in\Z_+$,
\begin{equation}
\mu\left(\left\{\max_{\substack{N_k\leq N < N_{k+1} \\ N\in I_R}}
|M_N^Tf - M_{N_{k+1}}^Tf|>\eps\right\}\right) > \eps.
\label{eq:measure}
\end{equation}
But if $f\in L^2(\mu)$, Markov's inequality and Theorem \ref{thm:oscillation}
give
\begin{align*}
\sum_{k=1}^\infty
\mu\left(\left\{\max_{\substack{N_k\leq N < N_{k+1} \\ N\in I_R}}
|M_N^Tf - M_{N_{k+1}}^Tf|>\eps\right\}\right)
&\leq
\frac{1}{\eps^2}
\sum_{k=1}^\infty \left\|
\max_{\substack{N_k\leq N < N_{k+1} \\ N\in I_R}}
|M_N^Tf - M_{N_{k+1}}^Tf|\right\|_{L^2(\mu)}^2 \\
&\ls_R \frac{1}{\eps^2} \, 
\|f\|_{L^2(\mu)}^2,
\end{align*}
which contradicts \eqref{eq:measure}.

Secondly, by the transference principle, it suffices to prove Theorem
\ref{thm:oscillation} in the case that $X=\Z[i]$, $\mu$ is counting measure
and $T$ is the standard shift, i.e. for $n,m\in\Z[i]$, $T^n(m)=n+m$. In fact, we focus on proving the following inequality 
\begin{align}\label{e:amenablereduction}
\sum_{k=1}^\infty \left\|
\max_{\substack{N_k\leq N < N_{k+1} \\ N\in I_R}}
|A_Nf - A_{N_{k+1}}f|\right\|_{\ell^2}^2
\ls_R
\|f\|_{\ell^2}^2.
\end{align}

For $N\geq1$ and $s$ such that $2^{s+1}\leq N^{1/20}$, we define multipliers
$V_{N,s}$ and $U_{N,s}$ on $\ell^2(\Z[i])$ by
$$
\wh{V_{N,s}}(\alpha)
=
\sum_{a/q \in \calR_s} \frac{\pi^2}{\N(q)} \eta(\sqrt{N}(\alpha - a/q)),
\quad
\wh{U_{N,s}}(\alpha)
=
\sum_{a/q \in \calR_s} \eta(\sqrt{N}(\alpha - a/q)).
$$
In addition, we define $U_{N,s}=V_{N,s}=0$ when $2^{s+1}>N^{1/20}$ and set
$$
V_N = \sum_{s=0}^\infty V_{N,s},
\quad
U_N = \sum_{s=0}^\infty U_{N,s}.
$$
Observe that these multipliers have the following properties:
\begin{enumerate}[label=(\arabic*)]
\item
If $N\geq 2M$, then $U_{M,s}U_{N,s}=U_{N,s}U_{M,s}=U_{N,s}$;
\item
For $f\in\ell^2(\Z[i])$, $V_{N,s}f = U_{N,s}\tilde{f}$, where
\begin{equation}
\wh{\tilde{f}}(\alpha) = \sum_{a/q\in\calR_s}
\frac{\pi^2}{\N(q)}\eta_s(\alpha-a/q)\wh{f}(\alpha).
\label{eq:f-tilde}
\end{equation}
\end{enumerate}

\begin{lemma}\label{lem:sq-function}
Let $\rho>1$. For each sequence of integers $\{N_k\}$ such that
$N_{k+1}\geq\rho N_k$, and each $f\in\ell^2(\Z[i])$ we have
$$
\sum_{k=1}^\infty
\|A_{N_k}f - V_{N_k}f\|_{\ell^2}^2
\ls_\rho
\|f\|_{\ell^2}^2.
$$
\end{lemma}

\begin{proof}
We need to prove the bound
$$
\left\| \left(\sum_{k=1}^\infty
|\wh{A_{N_k}} - \wh{V_{N_k}}|^2 \right)^{1/2}\right\|_\infty
\ls_\rho
1.
$$
Recall the definition of $K'_{N,s}$ in \eqref{eq:KNdef}. For simplicity, we redefine
$K'_{N,s}$ to be $0$ when $2^{s+1}>\sqrt{N}$ and set
$$
K_N = \sum_{s=0}^\infty K'_{N,s}.
$$

Fix $\alpha\in\T^2$. Then
\begin{align}
&\left(\sum_{k=1}^\infty |\wh{A_{N_k}}(\alpha) - \wh{V_{N_k}}(\alpha)|^2 \right)^{1/2} \\
& \qquad \leq
\left(\sum_{k=1}^\infty |\wh{A_{N_k}}(\alpha) - \wh{K_{N_k}}(\alpha)|^2 \right)^{1/2}
+\; \left(\sum_{k=1}^\infty |\wh{K_{N_k}}(\alpha) - \wh{V_{N_k}}(\alpha)|^2 \right)^{1/2} \\
& \qquad \leq
\left(\sum_{k=1}^\infty |\wh{A_{N_k}}(\alpha) - \wh{K_{N_k}}(\alpha)|^2 \right)^{1/2} 
+\; \sum_{s=0}^\infty\left(\sum_{k=1}^\infty |\wh{K'_{N_k,s}}(\alpha) - \wh{V_{N_k,s}}(\alpha)|^2 \right)^{1/2}.
\end{align}
By Lemma \ref{l:applemm}, the first term satisfies
$$
\sum_{k=1}^\infty |\wh{A_{N_k}}(\alpha) - \wh{K_{N_k}}(\alpha)|^2
\ls
\sum_{k=1}^\infty (\log N_k)^{-2}
\ls_\rho
\sum_{k=1}^\infty \frac{1}{k^2}.
$$
To bound the second term we note that if $|\wh{K'_{N_k,s}}(\alpha) - \wh{V_{N_k,s}}(\alpha)|\neq0$, there is a unique $a/q\in\calR_s$ such that
\begin{equation}
|\wh{K'_{N_k,s}}(\alpha) - \wh{V_{N_k,s}}(\alpha)| \\
=
\frac{\pi^2}{\N(q)}\left| \frac{1}{\pi N}\sum_{\substack{\N(n)\leq N}} 
e(\jap{n,\alpha-a/q})
- \eta(\sqrt{N_K}(\alpha - a/q))\right|.
\label{eq:KVS}
\end{equation}
Choose $k_0$ such that $N_{k_0 +1}^{-1} < \N(\alpha-a/q) \leq N_{k_0}^{-1}$.
If $k\leq k_0$ and $\eta(\sqrt{N_K}(\alpha - a/q))=1$ \eqref{eq:KVS} becomes
\begin{align}
\frac{\pi^2}{\N(q)}\left| \frac{1}{\pi N}\sum_{\substack{\N(n)\leq N}} 
e(\jap{n,\alpha-a/q}) - 1\right|
&\ls
\frac{1}{\N(q)}((N_k/N_{k_0})^{1/2} + N_k^{-1/2}) \\
&\leq
2^{-s}(\rho^{-(k_0-k)/2} + \rho^{-k/2}).
\end{align}
For the remaining $k\leq k_0$, we trivially bound \eqref{eq:KVS} by $2^{-s}$.
Note that there are at most $O((\log\rho)^{-1})$ such $k$. It follows that
\begin{equation}
\sum_{k\leq k_0} |\wh{K'_{N_k,s}}(\alpha) - \wh{V_{N_k,s}}(\alpha)|^2
\ls
2^{-2s} \left((\log\rho)^{-1} +
\sum_{k\leq k_0} (\rho^{-(k_0-k)/2} + \rho^{-k/2})\right)^2
\ls_\rho 2^{-2s}.
\label{eq:k-small}
\end{equation}
For $k>k_0$, we have that $\eta(\sqrt{N_K}(\alpha - a/q))=0$, so
\eqref{eq:KVS} becomes 
\begin{align}
\frac{\pi^2}{\N(q)}\left| \frac{1}{\pi N}\sum_{\substack{\N(n)\leq N}} 
e(\jap{n,\alpha-a/q})\right|
&\ls
\frac{1}{\N(q)}((N_k/N_{k_0})^{-3/4} + N_k^{-1/2}) \\
&\leq
2^{-s}(\rho^{-3(k-k_0)/4} + \rho^{-k/2}),
\end{align}
and so
\begin{equation}
\sum_{k>k_0} |\wh{K'_{N_k,s}}(\alpha) - \wh{V_{N_k,s}}(\alpha)|^2
\ls
2^{-2s} \left(\sum_{k> k_0} (\rho^{-3(k-k_0)/4} + \rho^{-k/2})\right)^2
\ls_\rho 2^{-2s}.
\label{eq:k-large}
\end{equation}
Combining \eqref{eq:k-small} and \eqref{eq:k-large} then summing over $s$
completes the proof.
\end{proof}

\begin{proof}[Proof of Theorem \ref{thm:oscillation}.]
As before, we have
\begin{align}
&\left(\sum_{k=1}^\infty \left\|
\max_{\substack{N_k\leq N < N_{k+1} \\ N\in I_R}}
|A_Nf - A_{N_{k+1}}f|\right\|_{\ell^2}^2\right)^{1/2}\\
& \qquad \ls
\sum_{s=0}^\infty
\left(\sum_{k=1}^\infty \left\|
\max_{\substack{N_k\leq N < N_{k+1} \\ N\in I_R}}
|V_{N,s}f - V_{N_{k+1},s}f|\right\|_{\ell^2}^2\right)^{1/2}
+\; \left(\sum_{k=1}^\infty
\|A_{N_k}f - V_{N_k}f\|_{\ell^2}^2\right)^{1/2}.
\end{align}
In light of Lemma \ref{lem:sq-function}, it only remains to bound
the first term on the right.

Take $f\in\ell^2(\Z[i])$ and define $\tilde{f}$ as in \eqref{eq:f-tilde}.
Then for each $s$ and $N_k\leq N < N_{k+1}$,
\begin{align}
|V_{N,s}f - V_{N_{k+1},s}f|
&=
|U_{N,s}\tilde f - U_{N_{k+1},s}\tilde f| \\
&\leq
|U_{N,s}(U_{N_{k-1},s}- U_{N_{k+1},s})\tilde f|
+
|(U_{N_{k+2},s}- U_{N_{k+1},s})\tilde f|.
\end{align}
Using this followed by Theorem \ref{t:multi} we see that
\begin{align}
\sum_{k=1}^\infty \left\|
\max_{\substack{N_k\leq N < N_{k+1} \\ N\in I_R}}
|V_{N,s}f - V_{N_{k+1},s}f|\right\|_{\ell^2}^2
&\ls
\sum_{k=1}^\infty \left\|
\max_{N\in I_R} |U_{N,s}(U_{N_{k},s}- U_{N_{k+1},s})\tilde f| \right\|_{\ell^2}^2 \\
&\ls
(1+\log s)^2\sum_{k=1}^\infty \left\|
(U_{N_{k},s}- U_{N_{k+1},s})\tilde f \right\|_{\ell^2}^2 \\
&\leq
(1+\log s)^2\left\|\sum_{k=1}^\infty 
|\wh{U_{N_{k},s}}- \wh{U_{N_{k+1},s}}|^2\right\|_\infty
\|\tilde f \|_{\ell^2}^2 \\
&\ls
(1+\log s)^2 2^{-2s}\|f\|_{\ell^2}^2.
\end{align}
In the last step, we used the estimate $\|\tilde f \|_{\ell^2}\leq
2^{-s}\|f\|_{\ell^2}$ and the fact that for each $\alpha\in\T^2$, there
are at most 2 values of $k$ such that
$|\wh{U_{N_{k},s}}(\alpha)- \wh{U_{N_{k+1},s}}(\alpha)|\neq0$,
and so the sum is uniformly bounded. Then summing over $s$ produces the desired result.
\end{proof}
\bigskip

\section*{Acknowledgements}
The authors would like to thank Ben Krause and Michael Lacey for several fruitful conversations.

\end{document}